\documentclass[12pt,fleqn]{article}
\usepackage{amsmath,amssymb,amsthm,enumerate}
\usepackage{graphics,graphicx,float,caption,subcaption}
\newtheorem{thm}{Theorem }
\newtheorem{lem}{Lemma}

\newtheorem{cor}{Corollary}
\theoremstyle{definition}
\newtheorem{defn}{Definition} 
\newtheorem{exmp}{Example} 
\newtheorem{remark}{Remark}
\renewcommand{\Re}{\operatorname{Re}}
\renewcommand{\Im}{\operatorname{Im}}

\author{}
\date{ }

\begin{document}
	\newcommand{\norm}[1]{\lVert #1\rVert}
	\newcommand{\mnorm}[1]{\lvert #1\rvert}
	\newcommand{\ml}[1]{E_{\alpha}(#1)}
	\newcommand{\img}[1]{\mathrm{Img}\{#1\}}
	\newcommand{\mltwo}[3]{E_{#1, #2}(t^{#1} #3)}
	\newcommand{\pder}[2]{\dfrac{\partial^{#1}}{\partial #2^{#1}}}
\begin{center}
	{\Large Analysis of intersections of trajectories of linear systems of fractional order\\}
	{\large Amey S. Deshpande\footnotemark[1]  \footnotemark[2],  \space Varsha  Daftardar-Gejji\footnotemark[3]  \footnotemark[4] \\and Palaniappan Vellaisamy\footnotemark[1]  \footnotemark[5]}
\end{center}
	\footnotetext[1]{Department of Mathematics, IIT Bombay, Mumbai-400076}
	\footnotetext[2]{Email: 2009asdeshpande@gmail.com, ameyd@math.iitb.ac.in}
	\footnotetext[3]{Department of Mathematics, Savitribai Phule Pune University, Pune - 411007}
	\footnotetext[4]{Email: vsgejji@gmail.com, vsgejji@unipune.ac.in}
	\footnotetext[5]{Email: pv@math.iitb.ac.in}

\begin{abstract}
	Present article deals with trajectorial intersections in linear fractional systems (`systems'). We propose a classification of intersections of trajectories in three classes \textit{viz.} trajectories intersecting at same time(EIST), trajectories intersecting at distinct times(EIDT) and self intersections of a trajectory. We prove a generalization of separation theorem for the case of linear fractional systems. This result proves existence of EIST. Based on the presence of EIST, systems are further classified in two types; Type I and Type II systems, which are analyzed further for EIDT. Besides constant solutions and limit-cycle behavior, a fractional trajectory can have nodal or cuspoidal intersections with itself. We give a necessary and sufficient condition for a trajectory to have such types of intersections.    
\end{abstract}
\section{Introduction}

Study of fractional differential equations have seen increasing interest due to their applications in diverse fields \cite{hilfer2000applications, newworld}. For a detailed introduction to fractional calculus and fractional differential equations, we refer readers to \cite{podlubny1999fractional, diethelm2010analysis}. For a brief survey of the work in fractional systems refer to \cite{stabsurvey, Stamova2017}.

Present article deals with an $n$-dimensional autonomous linear fractional system with Caputo fractional derivative (referred as system)
\begin{equation} \label{ivp0}
D^{\alpha} x(t) = A x(t), ~~ t \geq 0, ~ 0 < \alpha <1,~ A \in \mathbb{R}^{n \times n},
\end{equation}   and studies dynamics of its solution (referred as trajectory) for $0 < t < \infty$. 

The question of intersections of trajectories of a fractional systems has been dealt before \cite{diethelm2012volterra, agarwal2010survey, hayek1998extension, bonilla2007systems, diethelm2010analysis, cong2017generation}. Diethelm \textit{et al.} \cite{diethelm2010analysis} have proved that for one dimensional fractional system, two distinct trajectories do not intersect each other at the same time. This result is also known as \textit{separation theorem} for fractional systems. They have also observed that fractional trajectories can still intersect each other at distinct times because of their inherent non-local nature. Recently, Cong. \textit{et al.} \cite{cong2017generation} have dealt with this question and generalized separation theorem for higher dimensional triangular fractional systems. Moreover, they have also constructed an example of a fractional system for which separation theorem does not work. 

Pursuance to this we generalize separation theorem for linear fractional systems and investigate whether solutions of fractional system \eqref{ivp0} intersect each other. Further we propose a classification of types of intersections and for each type give existence result. 
Such study is important for deriving deeper insights and understanding of intrinsic dynamics of fractional systems. This should lead to successful modelling of some physical phenomena using fractional systems.

The rest of the article is organized as follows. Section \ref{preliminaries} introduces preliminaries from fractional calculus and notations used throughout this article. Section \ref{fracdyn} categorizes possible intersections in fractional systems in to three broad categories. Section \ref{zerotrajectory} deals with intersections of zero trajectory. Section \ref{externalintersections} explores two or more trajectorial intersections.  Section \ref{selfintersections} deals with intersection of single trajectory with itself. Section \ref{conclusions} summarizes findings and conclusions and outlines some directions of future research.  
    
\section{Preliminaries} \label{preliminaries}
In this section, we introduce some preliminaries from fractional calculus. For more details, we refer the readers to  \cite{podlubny1999fractional, diethelm2010analysis, daftardar2007analysis}.

\begin{defn}
	The Riemann-Liouville fractional integral of order $\alpha > 0$ of $f \in C[0, \infty)$ is defined as
	\begin{align}
	I^{\alpha}f(t)=\frac{1}{\Gamma(\alpha)}\int\limits_{0}^{t}\frac{f(\tau)}{(t-\tau)^{1-\alpha}}d\tau.
	\end{align}
\end{defn}
\begin{defn}
	The Caputo derivative of order $\alpha \in (k-1, k], ~ k \in \mathbb{N}$ of $f \in C^k(0, \infty)$ is defined as 
	\begin{equation}
	D^{\alpha}f(t) =
	\begin{cases}
	
	\frac{1}{\Gamma(k-\alpha)} ~\int_{0}^{t} (t-\tau)^{k-\alpha-1} f^{(k)}(\tau)\: d\tau,~~ &\alpha \in (k-1,k),  \\
	 f^{(k)}(t), ~~ & \alpha = k.
	 \end{cases}
	\end{equation}
\end{defn}

If $f:[0, \infty) \rightarrow \mathbb{R}^n$, where $f = (f_1, f_2, \cdots, f_n)$, $f_i: [0, \infty) \rightarrow \mathbb{R}$, then $D^{\alpha} f = (D^{\alpha}f_1, D^{\alpha} f_2, \cdots, D^{\alpha} f_n)$. 

Let $f:[0, \infty) \times \mathbb{R}^n \rightarrow \mathbb{R}^n$ and $\alpha > 0$, then
\begin{equation} \label{fde1}
D^{\alpha} x(t) = f(t, x(t)),
\end{equation} denotes a system of fractional differential equations. The system in \eqref{fde1} is autonomous if $f$ does not explicitly depend on $t$ and linear if $f$ is linear in $x(t)$. For $0 < \alpha \leq 1$, the system in \eqref{fde1} along with initial condition $x(0)= x_0 \in \mathbb{R}^n$ constitutes fractional initial value problem (IVP).

In this article, we restrict ourselves to the following IVP consisting of linear fractional autonomous system, with $0 < \alpha <1$,
	\begin{equation} \label{ivp1}
	\begin{split}
	D^{\alpha} x(t) &= A x(t), ~A \in \mathbb{R}^{n \times n},\: x(t) \in \mathbb{R}^n,\: t \geq 0,\\
	x(0) &= x_0. 
	\end{split}
	\end{equation}
	

\begin{defn}
	The two parameter Mittag-Leffler(M-L) function is defined as 
	\begin{equation}
	 E_{\alpha, \beta}(z) = \sum_{k=0}^{\infty} \; \frac{z^{k}}{\Gamma(\alpha k +\beta)}, ~ z \in \mathbb{C}, ~ 0 < \alpha \leq 1, ~ \beta \in \mathbb{R}.
	\end{equation} For $\beta=1$, $E_{\alpha, 1}(z) = E_{\alpha}(z)$ is the classic Mittag-Leffler function. When $\alpha=1$ then $E_{1, 1}(z) = e^z$.
 \end{defn}

\begin{defn}
	The two parameter Mittag-Leffler matrix function (or M-L operator) is defined as 
	\begin{equation}
	E_{\alpha, \beta}(A) = \sum_{k=0}^{\infty} \; \frac{A^{k}}{\Gamma(\alpha k +\beta)}, ~ A \in \mathbb{R}^{n \times n}, ~ 0 < \alpha \leq 1, ~ \beta \in \mathbb{R}.
	\end{equation}
\end{defn}
\begin{thm}[Global existence and Uniqueness \cite{daftardar2004analysis}] \label{existence}
	The solution $x(t) = E_{\alpha}(t^{\alpha} A) x_0,$   $t \geq 0,$ is the unique solution of IVP in \eqref{ivp1}.
\end{thm}

\begin{defn}
The set $u(t;x_0) := \{x(t) ~|~ t \geq 0\}$ is called as \textbf{trajectory} of the IVP in \eqref{ivp1} starting at $x_0 \in \mathbb{R}^n$. 
\end{defn}

For a fixed point on the trajectory, we will use the notation $u(T;x_0) = x(T) = \ml{T^{\alpha}A} x_0$ alternatively.   

\begin{remark} \label{first}
	Note that Caputo fractional derivative for $0 < \alpha < 1$ do not satisfy chain rule in general and hence are not translation invariant \cite{podlubny1999fractional}.
\end{remark} 




Theorem \ref{existence}, along  with continuous dependence on initial data \cite{daftardar2004analysis},  together shows that fractional initial value problem in \eqref{ivp1} is (globally) \textit{well-posed}.

%
%
%
%
\section{Classification of intersections in linear fractional systems} \label{fracdyn}
\begin{defn}
	\textbf For $x_0, y_0 \in \mathbb{R}^n$, $x_0 \neq y_0$, and $t \geq 0$, trajectories $u(t;x_0), u(t; y_0)$ of \eqref{ivp1}, are said to \textit{intersect at point} $p \in \mathbb{R}^n$, if there exist $T, \tilde{T} \geq 0$, such that 
	\begin{equation}
	p = u(T; x_0) = u(\tilde{T}; y_0).
	\end{equation}
\end{defn}

\begin{remark}
	As highlighted in Remark \ref{first}, fractional derivative is not translation invariant and thus $u(T+t;x_0), u(T+t; y_0)$ are not solutions of IVP \eqref{ivp1} with $x(0) = p$. Therefore, unlike $\alpha =1 $ case, existence of intersections does not contradict uniqueness.
\end{remark}


Trajectorial intersections can be classified into three categories as follows.  
\begin{enumerate}
	 \item \textbf{External intersections at same time (EIST):} When two (or more) distinct trajectories intersect at point $p \in \mathbb{R}^n$ after traveling same amount of time $T > 0$ \textit{i.e.} $p=u(T; x_0) = u(T; y_0) \text{ and } ~ x_0 \neq y_0$.
	 \item \textbf{External intersections at different time (EIDT): } When two (or more) distinct trajectories intersect at point $p \in \mathbb{R}^n$ after traveling different amounts of time say $T, \tilde{T} \geq 0, ~ T \neq \tilde{T}$, that is $p=u(T; x_0) = u(\tilde{T}; y_0), ~ x_0 \neq y_0$.
	 \item \textbf{Self intersection:} When single trajectory intersects itself again in finite time, that is for $T, \tilde{T} \geq 0, ~ T \neq \tilde{T}$, $u(T;x_0) = u(\tilde{T};x_0)$.
\end{enumerate}

First we derive some results which are further applied.
\begin{lem} \label{selfc1}
	An eigenvalue of $E_{\alpha, \beta}(T^{\alpha} A), ~ T > 0 $, is of the form $E_{\alpha, \beta}(T^{\alpha} \lambda)$, where $\lambda \in \mathbb{C}$ is an eigenvalue of $A \in \mathbb{R}^{n \times n}$, $0 < \alpha < 1, ~ \beta \in \mathbb{R}$.

\end{lem}  
\begin{proof}	
	Let $\lambda \in \mathbb{C}$, and $v \in \mathbb{C}^n$ be such that $Av = \lambda v, ~ v \neq 0$. Then $A^{k} v = \lambda^k v$, holds for $k \in \mathbb{N}$. Hence we get
	\begin{equation*}
	E_{\alpha, \beta}(T^{\alpha} A) v = \sum_{k=0}^{\infty} \frac{T^{\alpha k }}{\Gamma(\alpha k +\beta)}\; (A^k v) = \sum_{k=0}^{\infty} \frac{T^{\alpha k }}{\Gamma(\alpha k +1\beta)}\; (\lambda^k v) = E_{\alpha, \beta}(T^{\alpha} \lambda) v.
	\end{equation*} As there are exactly $n$ complex eigenvalues (with multiplicities), the result follows. 
\end{proof}	 

\begin{lem} \label{invertability}
	The matrix $E_{\alpha}(t^{\alpha} A), ~t>0, ~A \in \mathbb{R}^{n \times n}$ is invertible \textit{if and only if} $\arg(\lambda) \neq \arg(\zeta_{\alpha})$, where $\zeta_{\alpha} \in \mathbb{C}$ is a zero of Mittag-Leffler function $E_{\alpha}(z)$ and $\lambda$ an eigenvalue of $A$.  
\end{lem}
\begin{proof}
	Note $E_{\alpha}(t^{\alpha}A)$ is an invertible matrix  $\Longleftrightarrow\det(E_{\alpha}(t^{\alpha}A)) \neq 0 \Longleftrightarrow$  none of the eigenvalues of $E_{\alpha}(t^{\alpha}A)$ are zero. Using Lemma \ref{selfc1}, this is equivalent to $E_{\alpha}(t^{\alpha} \lambda) \neq 0$,  where $\lambda \in \mathbb{C}$ is an eigenvalue of $A$.
	But this is possible if and only if $t^{\alpha} \lambda \neq \zeta_{\alpha}$, for $\zeta_{\alpha} \in \mathbb{C}$ a zero of Mittag-Leffler function $E_{\alpha}(z)$. Thus we get $\arg(\lambda) \neq \arg(\zeta_{\alpha})$ as required. 
\end{proof}

\begin{lem}\label{commutativity}
	For $T, \tilde{T} > 0$, 
	\begin{enumerate}
		\item $E_{\alpha}(T^{\alpha}A) \; E_{\alpha}(\tilde{T}^{\alpha}A) = E_{\alpha}(\tilde{T}^{\alpha}A) \; E_{\alpha}(T^{\alpha}A),$
		\item If $E_{\alpha}(\tilde{T}^{\alpha}A)$ is invertible matrix, then 
		\begin{equation*}
		E_{\alpha}(T^{\alpha}A) \; E_{\alpha}(\tilde{T}^{\alpha}A)^{-1} = E_{\alpha}(\tilde{T}^{\alpha}A)^{-1} \; E_{\alpha}(T^{\alpha}A).
		\end{equation*}
	\end{enumerate}
\end{lem}
\begin{proof}
Follows from the definitions.
\end{proof}

\section{Intersections with zero trajectory} \label{zerotrajectory}
We deal with the case of intersections with zero trajectory in this section. Theorem  \ref{zero1} and Theorem  \ref{zero2} give necessary and sufficient conditions for a non-zero trajectory to intersect origin in finite time. 

\begin{thm} \label{zero1}
	If a non-zero trajectory $x(t) = u(t;x_0), ~ x_0 \neq 0, ~ t >0$ of IVP \eqref{ivp1} intersects origin at time $T > 0$, then $x_0 \in \ker\left\lbrace \ml{T^{\alpha} A} \right\rbrace $ and at least one of the eigenvalues of $A$, say $\lambda_k$, satisfies $\arg(\lambda_k) = \arg(\zeta_{\alpha, k})$, where $\zeta_{\alpha, k}, ~k \in \mathbb{N}$ denotes a zero of $\ml{z}$. 
\end{thm}
\begin{proof}
	Suppose $x(T) = 0$ for some $T > 0$. Then $\ml{T^{\alpha}A} x_0 = 0$ and so $x_0 \in \ker\{\ml{T^{\alpha}A}\}$. Further, as $x_0 \neq 0$, $\ml{T^{\alpha} A}$ is not invertible. Thus, by Lemma \ref{invertability}, there exist a zero $\zeta_{\alpha, k} \in \mathbb{C}$ of $\ml{z}$ such that $\frac{\zeta_{\alpha, k}}{T^{\alpha}}$ is one of the eigenvalue of $A$. Hence the result follows. 
\end{proof}

\begin{thm} \label{zero2}
If for some zero $\zeta_{\alpha, k}, ~ k \in \mathbb{N}$ of $\ml{z}$ and an eigenvalue $\lambda_k$ of $A$,  $\arg(\lambda_k) = \arg(\zeta_{\alpha, k})$ holds , then there exists a $T>0$ such that $\ker\{\ml{T^{\alpha}A}\} \neq \{0\}$ and for any $x_0 \in \ker\{\ml{T^{\alpha}A}\}$, $u(T;x_0) = 0$.
\end{thm}
\begin{proof}
Since $\arg(\lambda_k) = \arg(\zeta_{\alpha, k})$, there exists $T>0$ such that  $ \lambda_k = \frac{\zeta_{\alpha,k}}{T^{\alpha}}$ and by Lemma \ref{invertability}, $\det(\ml{T^{\alpha}A}) = 0$. That is $\ker\{\ml{T^{\alpha}A}\} \neq \{0\}$. Let $x_0 \in \ker\{\ml{T^{\alpha}A}\}, ~ x_0 \neq 0$. Then $x(T) = u(T;x_0) = \ml{T^{\alpha}A}\;x_0 = 0$.
\end{proof}
\begin{cor}
	Each eigenvalue $\lambda$ of $A$ satisfies $\arg(\lambda) \neq \arg(\zeta_{\alpha})$, where $\zeta_{\alpha} \in \mathbb{C}$ is a zero of $\ml{z}$ if and only if trajectory $x(t) = u(t;x_0),~x_0 \neq0$ of IVP \eqref{ivp1} satisfies $x(t) \neq 0, ~ 0< t < \infty$. 
\end{cor}

In view of these results we can completely characterize all possible intersections with zero trajectory in linear fractional systems.  

\section{External intersections}\label{externalintersections}

The following theorem gives a necessary condition for linear fractional system to have \textit{EIST}.

\begin{thm} \label{sametime1}
	For $x_0,\; y_0 \in \mathbb{R}^n, ~ x_0 \neq y_0$ and $T > 0$, if trajectories of IVP \eqref{ivp1} satisfy $u(T;x_0)=u(T;y_0)$, then  at least one of the eigenvalues of $A$, say $\lambda_k$, satisfies $\arg(\lambda_k) = \arg(\zeta_{\alpha, k})$, where $\zeta_{\alpha, k}, ~ k \in \mathbb{N}$ is a zero of $\ml{z}$.
\end{thm}
\begin{proof}
	Let $z_0 = x_0-y_0 \neq 0$ and consider $z(t) = u(t;z_0)$. Due to linearity, $z(t)$ is non-zero trajectory of IVP of \eqref{ivp1}. Further 
	\begin{equation*}
		z(T)= u(T;z_0) = \ml{T^{\alpha}A}z_0 = \ml{T^{\alpha}A}(x_0-y_0) = 0.
	\end{equation*} This implies $\det(\ml{T^{\alpha}A})=0$ and  by Lemma \ref{invertability}, $\frac{\zeta_{\alpha, k}}{T^{\alpha}}$ is one of the eigenvalues of $A$.
\end{proof}

\begin{cor}[Generalized separation theorem for linear systems] \label{sametimecor}
	For each eigenvalue $\lambda$ of $A$ and zero $\zeta_{\alpha}$ of $\ml{z}$, if $\arg(\lambda) \neq \arg(\zeta_{\alpha})$ \
	then precisely one trajectory of IVP of \eqref{ivp1} crosses $p \in \mathbb{R}^n$ at time $t=T >0 $. 
\end{cor}
\begin{proof}
	As $\arg(\lambda) \neq \arg(\zeta_{\alpha})$, by Lemma \ref{invertability}, $\ml{T^{\alpha}A}, ~ T>0$ is invertible and hence $\mathrm{Img}\{\ml{T^{\alpha}A}\}=\mathbb{R}^n$. Thus, for $p \in \mathbb{R}^n$, there exists a $x_T \in \mathbb{R}^n$ such that $\ml{T^{\alpha}A}x_T = p$. The uniqueness of this $x_T$ follows from Theorem \ref{sametime1}.
\end{proof}

Using Theorem \ref{sametime1}, we can classify the systems into following two categories.
\begin{enumerate}[ \textbf{Type} I:]
	\item Each eigenvalue $\lambda$ of $A$ satisfies $\arg(\lambda) \neq \arg(\zeta_{\alpha})$, $\zeta_{\alpha} \in \mathbb{C}$, being zero of $\ml{z}$.
\item at least one eigenvalue of $A$ say  $\lambda_k$ satisfies $\arg(\lambda_k) = \arg(\zeta_{\alpha, k})$ for some zero $\zeta_{\alpha, k}, ~ k \in \mathbb{N}$ of $\ml{z}$.
\end{enumerate}

As a consequence of Theorem \ref{sametime1}, \textbf{Type I} systems are free from EIST. Note that all one-dimensional linear systems, as shown by Diethelm \cite{diethelm2010analysis} and triangular linear systems, as shown by Cong \textit{et al.}\cite{cong2017generation} are strict subsets of \textbf{Type I} systems.      

We analyze \textbf{Type I} systems further for EIDT.  Let $x_0 \in \mathbb{R}^n$ be fixed and $x(t)=u(t;x_0)$ denotes its trajectory. The aim is to find all possible trajectories which will intersect $u(t;x_0)$ at distinct times. By Theorem \ref{sametime1} and corollary \ref{sametimecor} , for each $t\geq 0$, we can find a unique $x_t \in \mathbb{R}^n$, such that $u(t;x_t)=x_0$. Let $\gamma_{x_0}: [0, \infty) \rightarrow \mathbb{R}^n$ be defined as $\gamma_{x_0}(t) = x_t$. Then $\gamma_{x_0}$ is connected continuous curve in $\mathbb{R}^n$ with $\gamma_{x_0}(0) = x_0$. Thus $\gamma_{x_0}$ represents collection of all points in $\mathbb{R}^n$ whose trajectories intersect $u(t;x_0)$ at point $x_0$ in distinct times. Note that for $\alpha=1$ case, $\gamma_{x_0}$ is the reverse time evolution of trajectory $u(t; x_0)$. Thus we call this curve as inverse curve.
\begin{defn}
	For point $x_0 \in \mathbb{R}^n$ and $x(t) = u(t;x_0), ~ t \geq 0$ being solution of IVP \eqref{ivp1} of \textbf{Type I} system, the inverse curve of $x_0$ i.e. $\gamma_{x_0}$  is defined as   
\begin{equation} 
\gamma_{x_0}(t) = \ml{t^{\alpha} A}^{-1} x_0, ~~ t \geq 0.
\end{equation}
\end{defn}
 
For $T > 0$, if $p = u(T;x_0)$, then for $t \geq 0$, 
  \begin{align*}
  \gamma_p(t)&=\ml{t^{\alpha} A}^{-1} p =\ml{t^{\alpha}A}^{-1}\;[\ml{T^{\alpha} A} x_0], \\
  &= \ml{T^{\alpha} A}\;[ \ml{t^{\alpha}A}^{-1}x_0] =\ml{T^{\alpha} A} \;\gamma_{x_0}(t).
  \end{align*}
  
Thus, we get
\begin{equation} \label{gammaevolve}
\gamma_p(t) = u(T; \gamma_{x_0}(t)) 
\end{equation} 

Define the set 
\begin{equation} \label{S}
S= \bigcup_{T \geq 0} \;\bigcup_{t \geq 0}\; \gamma_{x(T)}(t)
\end{equation}
S represents collection of all points whose trajectories will intersect $u(t;x_0)$ in distinct times. Thus we have proved following result about EIDT.
\begin{thm}
	Let $x_0 \in \mathbb{R}^n$ and $x(T) = u(T;x_0), ~ T\geq 0$ be a solution of IVP \eqref{ivp1} of \textbf{Type I} system. For any $x \in S$, we can find unique pair $T, t \geq 0$, such that $T \neq t$ and 
	\begin{equation}
	u(t;x) = u(T;x_0).
	\end{equation} Further, the points in $S$ are precisely the points having this property.
\end{thm} 

\begin{exmp}
	 Consider the IVP given in \eqref{ivp1} with $A = \begin{bmatrix}0 & 1\\-1 & 0 \end{bmatrix}$ and $x_0 = (2,1)^t \in \mathbb{R}^2$. The trajectory $u(t;x_0)$ in this case is given as 
	 $u(t;x_0) = \ml{t^{\alpha} A} x_0, ~ t \geq 0$, where 
	 \begin{equation*}
	 \ml{t^{\alpha}A} = \begin{bmatrix}
	 \Re(\ml{t^{\alpha} \;i}) & \Im(\ml{t^{\alpha} \;i}) \\
	 -\Im(\ml{t^{\alpha} \;i}) & \Re(\ml{t^{\alpha} \;i})
	 \end{bmatrix}.
	 \end{equation*}
	 
	 Now $\det(\ml{t^{\alpha}A}) = \mnorm{\ml{t^{\alpha} \; i}}^2 \neq 0$, since $t^{\alpha} i \in \mathbb{C}$ is not a zero of $\ml{z}$. Therefore, the curve $\gamma_{x_0}$ is well-defined and given as
	 \begin{equation*}
	 \img{\gamma_{x_0}} = \{\ml{t^{\alpha}A}^{-1} x_0~| t \geq 0\} \subset \mathbb{R}^n.
	 \end{equation*} 
	 
	 Figure  \ref{fig:gamma-curve-1} shows trajectory $u(t;x_0)$(green),  the curve $\gamma_{x_0}$ (red) and the evolution of various trajectories starting at points of $\gamma_{x_0}$(grey). It is clear that all these trajectories will intersect $x_0$ in finite distinct times. Figure  \ref{fig:gamma-curve-2} shows the evolution of the curve $\gamma_{x_0}$ along the trajectory $u(t;x_0)$(green). The curves $\gamma_q, \gamma_r$ for points $q = u(0.5;x_0)$ and $r=u(1.2;x_0)$ are obtained by trajectorially evolving curve $\gamma_{x_0}$(red) for time $t=0.5$ and $t=1.2$ respectively.
	 
	 \begin{figure}[H]
	 	\centering
	 	\includegraphics[width=0.7\linewidth]{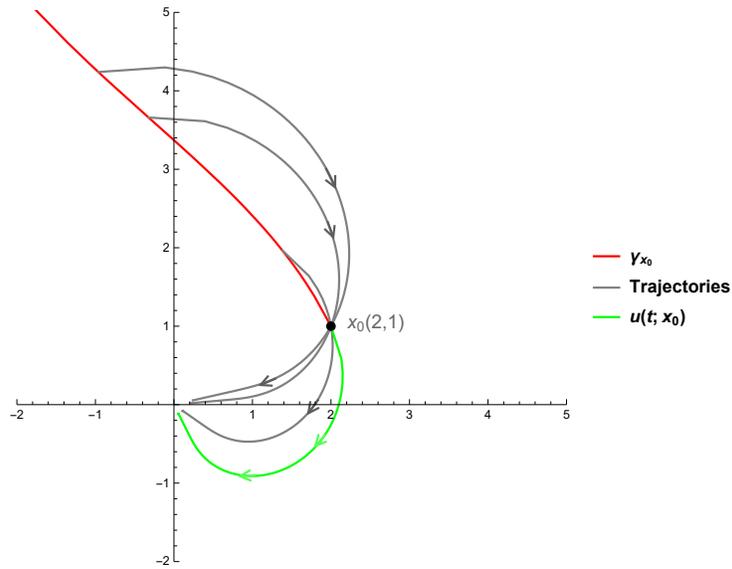}
	 	\caption{Figure shows curve $\gamma_{x_0}$ and evolution of various trajectories starting on $\gamma_{x_0}$.}
	 	\label{fig:gamma-curve-1}
	 \end{figure}
	 \begin{figure}[H]
	 	\centering
	 	\includegraphics[width=0.7\linewidth]{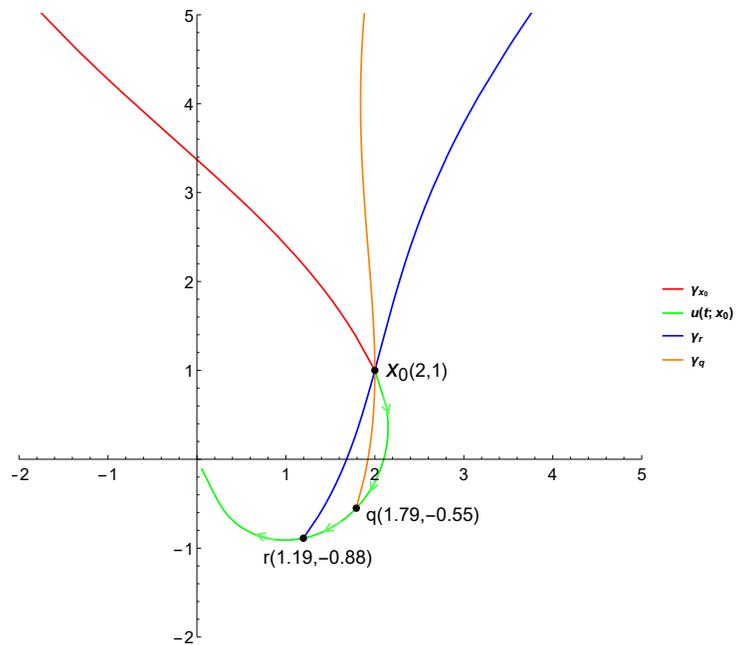}
	 	\caption{Figure shows evolution of the curve $\gamma_{x_0}$ along the trajectory $u(t;x_0)$}
	 	\label{fig:gamma-curve-2}
	 \end{figure}
	 	    
\end{exmp}
   
\textbf{Type II} systems have both EIST and EIDT.
\begin{thm} \label{sametime2}
	Let $T>0, ~ k \in \mathbb{N}$ and $\zeta_{\alpha, k} \in \mathbb{C}$ be a zero of $\ml{z}$.  If $\frac{\zeta_{\alpha, k}}{T^{\alpha}}$ is an eigenvalue of $A$, then
		\begin{enumerate}[(i)]
			\item $\ker\{\ml{T^{\alpha}A}\} \neq \{0\}$; 
			\item for $z_0 \in \ker\{\ml{T^{\alpha}A}\}$ and $p \in \mathrm{Img}\{\ml{T^{\alpha}A}\}$, there exists $x_0 \in \mathbb{R}^n$ such that $u(T; x_0 + z_0) = p$.
		\end{enumerate}
	
\end{thm}

\begin{proof}
	If $\mathrm{Img}\{\ml{T^{\alpha}A}\} = \{0\}$, then $p=0$, and in view of Theorem \ref{zero2} result follows. Thus, we assume that $p \neq 0$. As $p \in \mathrm{Img}\{\ml{T^{\alpha}A}\}, ~ \exists \;x_0 \in \mathbb{R}^n,~ x_0 \neq 0$, such that $\ml{T^{\alpha}A} x_0 = p$.
	
	 Given that $\frac{\zeta_{\alpha, k}}{T^{\alpha}}$ is an eigenvalue of $A$, by Lemma \ref{invertability} we have $\det(\ml{T^{\alpha}A})=0$ or equivalently $\ker\{\ml{T^{\alpha}A}\} \neq \{0\}$. For a non-zero $z_0 \in \ker\{\ml{T^{\alpha}A}\},$ let $y_0:= x_0 + z_0$. Now $y_0 \neq 0$, since $y_0=x_0 + z_0 = 0$, implies $x_0 \in \ker\{\ml{T^{\alpha}A}\}$, and thus $p = \ml{T^{\alpha}A}x_0 = 0$ which is a contradiction. Further, trajectory $y(t)=u(t;y_0)$ of IVP in \eqref{ivp1} satisfies
	 \begin{equation*}
	 y(T)=\ml{T^{\alpha}A}y_0 = \ml{T^{\alpha}A}(x_0 + z_0) = \ml{T^{\alpha}A}x_0 + \ml{T^{\alpha}A}z_0 = p.
	 \end{equation*} 
\end{proof}

Theorem \ref{sametime2} implies that, due to presence of EIST in \textbf{Type II} systems, all trajectories collapse onto $\img{\ml{T^{\alpha}A}}$ at same time $t=T$(See Example \ref{exmp2} for illustration). Thus, $\img{ \ml{T^{\alpha}A} }$ is the space of all EIST points in \textbf{Type II} system.

Further we analyze \textbf{Type II} systems for EIDT and EIST.  In this section, we assume that $\frac{\zeta_{\alpha, k}}{T^{\alpha}}$  is an eigenvalue of matrix $A$, where $T>0$ is fixed and $\zeta_{\alpha, k} \in \mathbb{C}$ a zero of $\ml{z}$. 

\begin{defn} \label{def8}
	Let $x_0 \in \mathbb{R}^n$ and $x(t) = u(t; x_0), ~ t \geq 0$, be a solution of \textbf{Type II} system described above. Then the inverse curve of $x_0$ is defined as 
	\begin{equation}
	\gamma_{x_0}(t) = \ml{t^{\alpha} A}^{-1} x_0, ~~ t \geq0, ~ t \neq T.
	\end{equation}
\end{defn}

\begin{defn}
	For a \textbf{Type II} system described in Definition \ref{def8} and $x \in \mathbb{R}^n$ we define set  $H_{x, T} =\{y \in \mathbb{R}^n~ \bigg| ~ \ml{T^{\alpha}A} y = x\}.$ In particular
	\begin{equation}
		H_{x, T}=\begin{cases}
 \phi, ~~ &x \notin \img{\ml{T^{\alpha}A}}, \\
x_0 + \ker\{\ml{T^{\alpha}A}\}, ~~&x \in \img{\ml{T^{\alpha}A}} ,
	\end{cases} 
	\end{equation} where $x_0 \in \mathbb{R}^n$ such that $\ml{T^{\alpha} A}x_0 = x$.
\end{defn}

The set  $\bigcup_{t \neq T} \gamma_{x_0}(t) \bigcup H_{x_0, T}$ is collection of all points in $\mathbb{R}^n$ whose trajectories will intersect $x_0$ in finite time.

 \begin{lem} \label{singular1}
 Consider a \textbf{Type II} system as described above, let $q = u(\tilde{T}; x_0)$, where $\tilde{T} > 0$. Then  
 \begin{enumerate}[(i)]
 	\item $\gamma_q(t) = \ml{\tilde{T}^{\alpha}A}\gamma_{x_0}(t),~ t\neq T$. \label{one}
 	\item $H_{q, T} = \ml{\tilde{T}^{\alpha} A}\; H_{x_0, T}, ~ \tilde{T} \neq T$. \label{second}
 \end{enumerate}
\end{lem} 
\begin{proof}
	Also for $t \neq T$, we have
	\begin{align*}
	\gamma_q(t)&=\ml{t^{\alpha} A}^{-1} q = \ml{t^{\alpha}A}^{-1}\;[\ml{\tilde{T}^{\alpha} A} x_0] \\
	&= \ml{\tilde{T}^{\alpha} A}\;[ \ml{t^{\alpha}A}^{-1}x_0] =\ml{\tilde{T}^{\alpha} A} \gamma_{x_0}(t),
	\end{align*} which proves Part \eqref{one}.
\end{proof}

For $\tilde{T} \neq T$
\begin{align*}
y \in H_{q, T} &\iff \ml{T^{\alpha} A} y = q \iff \ml{T^{\alpha} A} y  = \ml{\tilde{T}^{\alpha} A} x_0 \\
&\iff \ml{\tilde{T}^{\alpha} A}^{-1} \ml{T^{\alpha}A} y = x_0 \\ 
&\iff \ml{T^{\alpha} A} [\ml{\tilde{T}^{\alpha} A}^{-1} y] = x_0.
\end{align*} Thus $\ml{\tilde{T}^{\alpha}A}^{-1} y  \in H_{x_0, T}$, \textit{i.e.} $H_{x_0, T} \neq \phi$. But, this is equivalent to $y \in \ml{\tilde{T}^{\alpha}A} H_{x_0, T}$, which proves Part \eqref{second}.

\begin{lem}
For a \textbf{Type II} system as described above, let $p=u(T;x_0)$. Then  $H_{p, T} = x_0 + \ker\{\ml{T^{\alpha}A}\} $. 
\end{lem}
\begin{proof}
Note that
\begin{align*}
y \in H_{p, T} &\iff \ml{T^{\alpha} A} y = p \iff \ml{T^{\alpha} A} y  = \ml{T^{\alpha} A} x_0 \\
&\iff \ml{T^{\alpha}A} (y-x_0) = 0 \iff y-x_0 \in \ker\{\ml{T^{\alpha}A}\} \\
&\iff y \in x_0 + \ker\{\ml{T^{\alpha}A}\}.
\end{align*}
\end{proof}

Define set 
\begin{equation} \label{S2}
S= \left[ \bigcup_{\tilde{T} \geq 0} \;\bigcup_{t \neq T}\; \gamma_{x(\tilde{T})}(t)\right] \bigcup \left[ \bigcup_{\tilde{T} \geq 0}\; H_{x(\tilde{T}), T}\right].
\end{equation}
$S$ represents the collection of all points whose trajectories will intersect $u(t;x_0)$. 
 
\begin{exmp}
For $0 < \alpha < 1$,  $T > 0$ and $\zeta_{\alpha, k} \in \mathbb{C}$, a zero of $\ml{z}$, let
\begin{equation*}
A= \begin{bmatrix}
\Re(\frac{\zeta_{\alpha, k}}{T^{\alpha}}) & \Im(\frac{\zeta_{\alpha, k}}{T^{\alpha}}) \\
-\Im(\frac{\zeta_{\alpha, k}}{T^{\alpha}}) & \Re(\frac{\zeta_{\alpha, k}}{T^{\alpha}})
\end{bmatrix}.
\end{equation*} Then solution of IVP in \eqref{ivp1} is given as $u(t;x_0) = \ml{t^{\alpha} A} x_0, ~ t \geq 0,~ x_0 \in \mathbb{R}^n$, where 
\begin{equation*}
\ml{t^{\alpha}A} = 
\begin{bmatrix}
\Re(\ml{\left(\frac{t}{T} \right)^{\alpha} \zeta_{\alpha, k} }) &  \Im(\ml{\left(\frac{t}{T} \right)^{\alpha} \zeta_{\alpha, k} }) \\
-\Im(\ml{\left(\frac{t}{T} \right)^{\alpha} \zeta_{\alpha, k} }) & \Re(\ml{\left(\frac{t}{T} \right)^{\alpha} \zeta_{\alpha, k} })
\end{bmatrix}.
\end{equation*}

For $t=T$, $\ml{T^{\alpha}A} = 0$, and thus $\ker\{\ml{T^{\alpha}A}\} =  \mathbb{R}^n$ and $\img{\ml{T^{\alpha}A}} = \{0\}$. 

Therefore for any $x_0 \neq 0$, $S_{x_0, T} = \phi$. Further $p = u(T;x_0) = 0$, \textit{i.e.} every trajectory will cross origin at time $t=T$. And $S_{0 , T} = x_0 + \ker\{\ml{T^{\alpha}A}\} = \mathbb{R}^n$.  
\end{exmp}

\begin{exmp}\label{exmp2}
For $0 < \alpha < 1$,  $T > 0$ and $\zeta_{\alpha, k} \in \mathbb{C}$ being a zero of $\ml{z}$, let
\begin{equation*}
A= \begin{bmatrix}
\Re(\frac{\zeta_{\alpha, k}}{T^{\alpha}}) & \Im(\frac{\zeta_{\alpha, k}}{T^{\alpha}}) & 0\\
-\Im(\frac{\zeta_{\alpha, k}}{T^{\alpha}}) & \Re(\frac{\zeta_{\alpha, k}}{T^{\alpha}}) & 0 \\
0 & 0 & -1
\end{bmatrix}.
\end{equation*} Then the solution of IVP in \eqref{ivp1} is given as $u(t;x_0) = \ml{t^{\alpha} A} x_0, ~ t \geq 0,~ x_0 \in \mathbb{R}^n$, where 
\begin{equation*}
\ml{t^{\alpha}A} = 
\begin{bmatrix}
\Re(\ml{\left(\frac{t}{T} \right)^{\alpha} \zeta_{\alpha, k} }) &  \Im(\ml{\left(\frac{t}{T} \right)^{\alpha} \zeta_{\alpha, k} }) & 0 \\
-\Im(\ml{\left(\frac{t}{T} \right)^{\alpha} \zeta_{\alpha, k} }) & \Re(\ml{\left(\frac{t}{T} \right)^{\alpha} \zeta_{\alpha, k} }) & 0\\
0 & 0 & \ml{-t^{\alpha}}
\end{bmatrix}.
\end{equation*} Now $\det(\ml{T^{\alpha}A}) = 0$, and 
\begin{align*}
\ker\{\ml{T^{\alpha}A}\} &= \{(x, y, 0)~|~ x, y \in \mathbb{R}\}, \\
\img{\ml{T^{\alpha}A}} &= \{(0, 0, z)~|~z \in \mathbb{R}\}. 
\end{align*} For any $x_0 = (x_1,y_1,z_1) \in \mathbb{R}^3$, $p = u(T;x_0) = \ml{T^{\alpha}A} x_0 = (0, 0, c)$, where $c = \ml{-T^{\alpha}}z_1$. Thus, any trajectory starting on plane $z=z_1$ will intersect $z$-axis at time $t=T$ in the point $p \in \mathbb{R}^3$.  

Thus for any $x_0 = (x_1,y_1,z_1) \in \mathbb{R}^3$, not on $z$-axis, $S_{x_0, T} = \phi$. And for $p=(0, 0, c), ~ c \in \mathbb{R}$ set $S_{p, T} = \{(x_1, y_1, z_1) ~|~ x_1, y_1 \in \mathbb{R}, ~ z_1 = \frac{c}{\ml{-T^{\alpha}}}\}$.
\end{exmp}
\section{Self Intersections}\label{selfintersections}

A constant solution has self intersections. Kaslik \textit{et al.} \cite{kaslik2012nonlinear} have proved that there are no non-constant periodic solutions of class $C^1$ for fractional systems. Although fractional systems can have limit-cycle behavior, whenever $\mnorm{\arg(\lambda)} \geq \frac{\pi \alpha}{2}$ for all eigenvalues $\label{lambda}$ of $A$ and for those eigenvalues which satisfy $ \mnorm{\arg(\lambda)} = \frac{\alpha \pi }{2}$, geometric multiplicity is one \cite{stabsurvey}. Besides these, fractional trajectories can have following \textit{non-regular} types of self-intersections. 

\begin{defn}[See \cite{hilton1920plane}]
	A point $p \in \mathbb{R}^n$ is called as \textbf{point of multiple contact} (multiple point) of a non-constant trajectory $x(t) = u(t;x_0)$ of IVP in \eqref{ivp1} if $x(T) = p$, for some $T > 0$ and
	\begin{equation}
	\frac{d}{dt}x(t)\bigg|_{t = T} = ~E_{\alpha, 0}(T^{\alpha}A) x_0 = 0.
	\end{equation}
\end{defn}

At a multiple point, the trajectory will have two or more tangents. A standard double point is either \textit{cusp} or \textit{node} (see \cite{hilton1920plane}).

\begin{thm} \label{selfthm}
	A non-constant trajectory $x(t) = u(t; x_0)$ of IVP in \eqref{ivp1} has a multiple point at some $x(T) = p \in \mathbb{R}^n, ~ T>0$, if and only if, $x_0 \in \ker\{E_{\alpha, 0}(T^{\alpha}A)\}$ and there exists an eigenvalue $\lambda$ of $A$ which is of the form $\lambda = \frac{\eta_{\alpha, k}}{T^{\alpha}}$, where $\eta_{\alpha, k}, ~ k \in \mathbb{N}$ is a zero of $E_{\alpha, 0}(z)$. 
\end{thm} 
\begin{proof}
	Let $p = \ml{T^{\alpha} A} x_0 \in \mathbb{R}^n$ be a multiple point of trajectory $x(t)$. Hence, $E_{\alpha, 0}(T^{\alpha} A)x_0 = 0$. Then  $x_0 \in \ker\{E_{\alpha, 0}(T^{\alpha}A)\}$ and $x_0 \neq 0$, since $x(t)$ is non-constant. Therefore, $\det(E_{\alpha, 0}(T^{\alpha}A)) = 0$ \textit{i.e.}  $E_{\alpha, 0}(T^{\alpha}A)$ has at least one zero eigenvalue. By Lemma \ref{selfc1}, we get $E_{\alpha, 0}(T^{\alpha}\lambda) = 0$, for some eigenvalue  $\lambda$ of $A$. This proves the implication. The converse is proved by retracing the above steps in reverse direction.
\end{proof}

\begin{remark}
	Recently Bhalekar \textit{et al.} \cite{bhalekar2018self} have numerically found that for a 2-dimensional linear fractional systems, self-intersection occurs in region $\mnorm{\arg(\lambda_{\pm})} = \frac{\alpha \pi}{2} + \epsilon$, for sufficiently small $\epsilon > 0$ and  $\lambda_{\pm}$ being eigenvalues of system. This region comes as a direct consequence of Theorem \ref{selfthm} and the fact that zeros $\eta_{\alpha, k} \in \mathbb{C}$ of $E_{\alpha, 0}(z)$ are located in $\mnorm{\arg(z)} < \frac{\alpha \pi}{2} + \epsilon$, for large enough $k \in \mathbb{N}$ (See the proof of Theorem 4.7 in \cite{gorenflo2014mittag} ). 
\end{remark}

We construct an example of linear fractional system having self intersecting node and cusp using Theorem \ref{selfthm}. 

\begin{exmp} \label{exmpselfc}
	Let $\alpha = \frac{1}{3}$ and $\lambda_1 \approx 2.21095 - i (1.60243)$, $\lambda_2 = 1.47895 + i(1.349246)$ be zeros of $E_{\alpha, 0} (z)$. For $i=1, 2$,  $A_i = \begin{pmatrix}
	\Re(\lambda_i) & \Im(\lambda_i) \\ -\Im(\lambda_i) & \Re(\lambda_i)
	\end{pmatrix}$, and $x_0 = (1, 0)$, consider the IVP in \eqref{ivp1}. Solutions in this case are given as $x_i(t) = \ml{t^{\alpha}A_i} x_0, ~ t \geq 0,~ i=1,2$, where 
	\begin{equation*}
	\ml{t^{\alpha}A_i} = \begin{pmatrix}
	\Re(\ml{t^{\alpha}\lambda_i}) & \Im(\ml{t^{\alpha}\lambda_i}) \\
	- \Im(\ml{t^{\alpha}\lambda_i}) & \Re(\ml{t^{\alpha}\lambda_i})
	\end{pmatrix}.
	\end{equation*}
	These trajectories are plotted in Figure  \ref{fig:finiteloop} confirms the existence of double points having cusp (Figure  \ref{cusp})  and self-intersecting loop (Figure  \ref{loop}) each.   
\begin{figure}[H]
		\begin{subfigure}{0.5\textwidth}
			\centering
			\includegraphics[width=\linewidth]{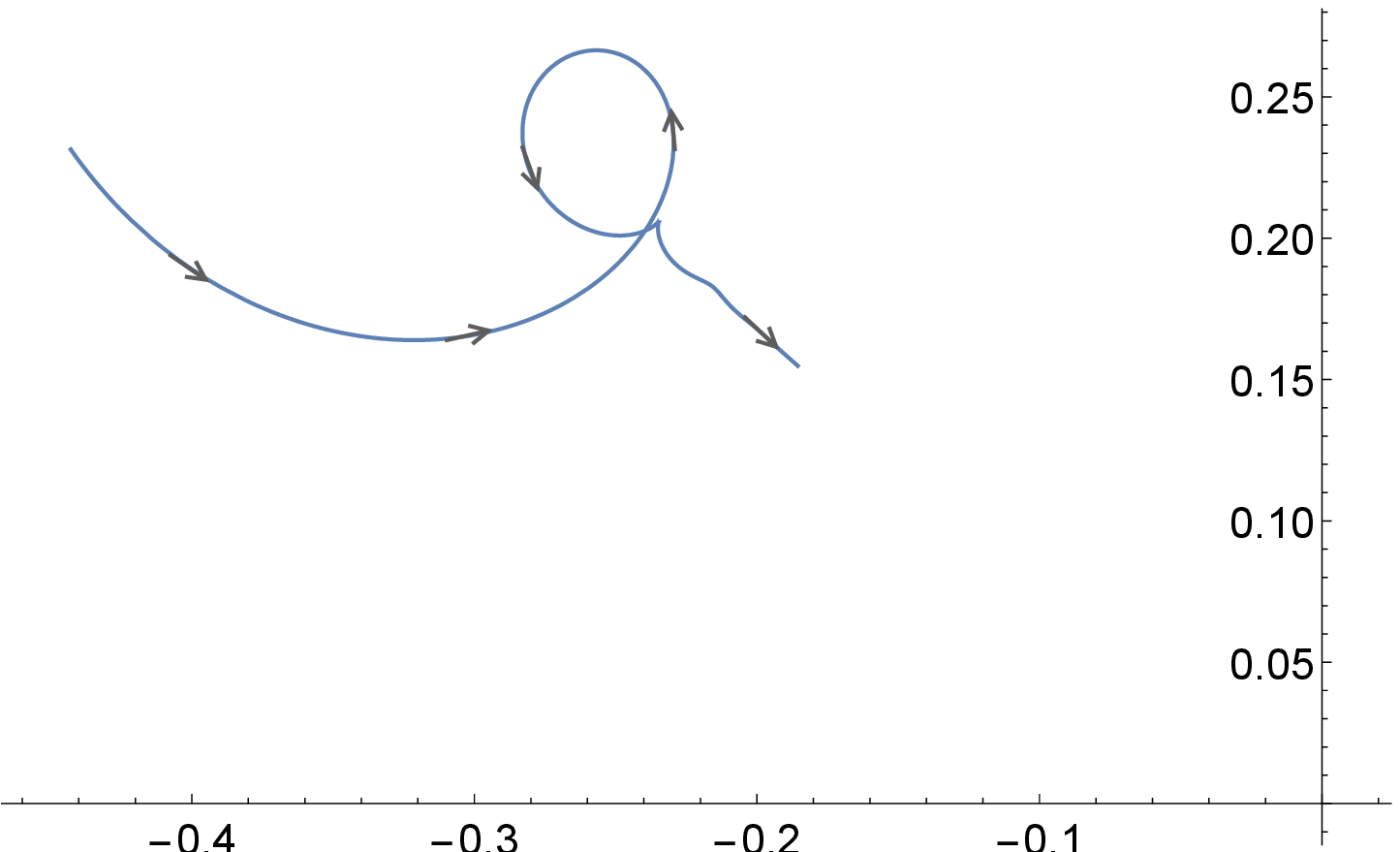}
			\subcaption{$x_1(t)$ for $\frac{1}{2} \leq t \leq 2$}
			\label{loop}
		\end{subfigure}\hspace{1cm}
		\begin{subfigure}{0.5\textwidth}
			\centering
			\includegraphics[width=\linewidth]{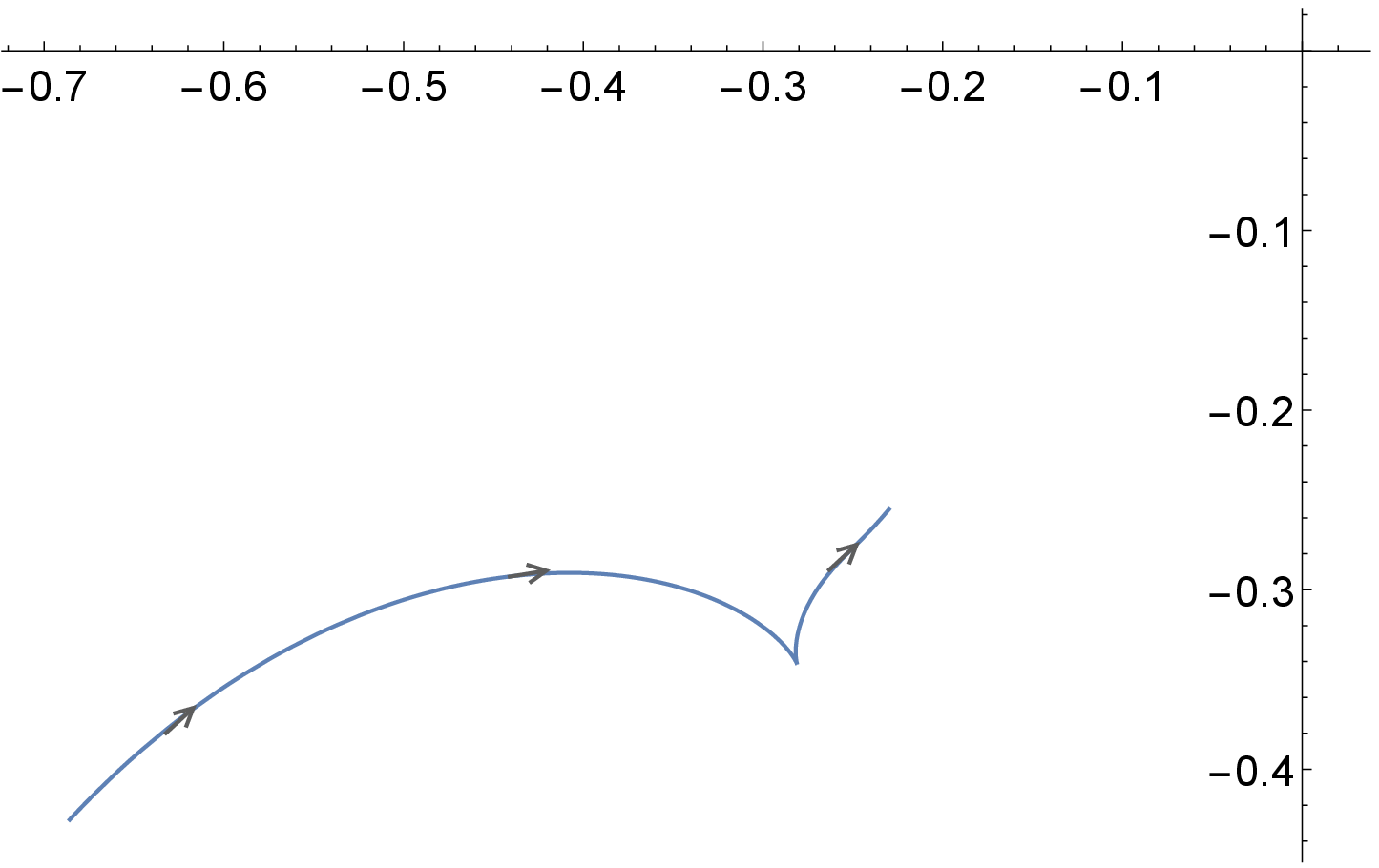}
			\subcaption{$x_2(t)$ for $\frac{1}{2} \leq t \leq 2$}
			\label{cusp}
		\end{subfigure}
\caption{Trajectory $x_1(t), x_2(t)$ of Ex. \ref{exmpselfc} showing self-intersection and cusp.}
\label{fig:finiteloop}
\end{figure}
\end{exmp}

\section{Conclusions and direction of future research}\label{conclusions}
In this article, we have classified trajectorial intersections in linear fractional systems into three broad categories \textit{viz.} external intersections occurring at same time(EIST), external intersections occurring at distinct times(EIDT), and self intersection. We have shown that the system will be free from EIST if and only if each eigenvalue $\lambda$ of a system satisfies $\arg(\lambda) \neq \arg{\zeta_{\alpha}}$, where $\zeta_{\alpha}$ is a zero of $\ml{z}$, which is a generalization of \textit{separation theorem} \cite{diethelm2010analysis} for the case of fractional linear systems. Existence of EIDT is an intrinsic feature of a fractional system. If  $\arg(\lambda) \neq \arg{\zeta_{\alpha}}$ holds, then there is unique trajectory intersecting point $p \in \mathbb{R}^n$ for each time $T > 0$, while if this condition fails, there are points where infinite trajectories intersect at the same time. We have shown that fractional trajectory can have cusps or nodes also. Further,  we have proved that these intersections occur if and only if $\arg(\lambda) = \arg{\eta_{\alpha}}$, where $\lambda$ is a eigenvalue of system and $\eta_{\alpha}$ is a zero of $E_{\alpha, 0 } (z)$. 

Further we would like to investigate whether similar characterization for EIST, EIDT and self intersections in fractional non-linear systems can be given. It would be an interesting question as to whether these features of fractional dynamics can be exploited to model some physical phenomena.


\begin{thebibliography}{10}
	
	\bibitem{agarwal2010survey}
	R.~P. Agarwal, M.~Benchohra, and S.~Hamani.
	\newblock A survey on existence results for boundary value problems of
	nonlinear fractional differential equations and inclusions.
	\newblock {\em Acta Applicandae Mathematicae}, 109(3):973--1033, 2010.
	
	\bibitem{bhalekar2018self}
	S.~Bhalekar and M.~Patil.
	\newblock Self-intersecting trajectories in fractional order dynamical systems.
	\newblock {\em arXiv preprint arXiv:1807.07731v1}, 2018.
	
	\bibitem{bonilla2007systems}
	B.~Bonilla, M.~Rivero, and J.~J. Trujillo.
	\newblock On systems of linear fractional differential equations with constant
	coefficients.
	\newblock {\em Applied Mathematics and Computation}, 187(1):68--78, 2007.
	
	\bibitem{cong2017generation}
	N.~Cong and H.~Tuan.
	\newblock Generation of nonlocal fractional dynamical systems by fractional
	differential equations.
	\newblock {\em Journal of Integral Equations and Applications}, 29(4):585--608,
	2017.
	
	\bibitem{daftardar2004analysis}
	V.~Daftardar-Gejji and A.~Babakhani.
	\newblock Analysis of a system of fractional differential equations.
	\newblock {\em Journal of Mathematical Analysis and Applications},
	293(2):511--522, 2004.
	
	\bibitem{daftardar2007analysis}
	V.~Daftardar-Gejji and H.~Jafari.
	\newblock Analysis of a system of nonautonomous fractional differential
	equations involving caputo derivatives.
	\newblock {\em Journal of Mathematical Analysis and Applications},
	328:1026--1033, 2007.
	
	\bibitem{diethelm2010analysis}
	K.~Diethelm.
	\newblock {\em The analysis of fractional differential equations: An
		application-oriented exposition using differential operators of Caputo type}.
	\newblock Springer Science \& Business Media, 2010.
	
	\bibitem{diethelm2012volterra}
	K.~Diethelm and N.~Ford.
	\newblock Volterra integral equations and fractional calculus: Do neighboring
	solutions intersect?
	\newblock {\em The Journal of Integral Equations and Applications},
	24(1):25--37, 2012.
	
	\bibitem{gorenflo2014mittag}
	R.~Gorenflo, A.~A. Kilbas, F.~Mainardi, and S.~V. Rogosin.
	\newblock {\em Mittag-Leffler functions, related topics and applications},
	volume~2.
	\newblock Springer, 2014.
	
	\bibitem{hayek1998extension}
	N.~Hayek, J.~Trujillo, M.~Rivero, B.~Bonilla, and J.~Moreno.
	\newblock An extension of picard-lindel{\"o}ff theorem to fractional
	differential equations.
	\newblock {\em Applicable Analysis}, 70(3-4):347--361, 1998.
	
	\bibitem{hilfer2000applications}
	R.~Hilfer.
	\newblock {\em Applications of fractional calculus in physics}.
	\newblock World Scientific, 2000.
	
	\bibitem{hilton1920plane}
	H.~Hilton.
	\newblock {\em Plane algebraic curves}.
	\newblock Clarendon Press, 1920.
	
	\bibitem{kaslik2012nonlinear}
	E.~Kaslik and S.~Sivasundaram.
	\newblock Nonlinear dynamics and chaos in fractional-order neural networks.
	\newblock {\em Neural Networks}, 32:245--256, 2012.
	
	\bibitem{stabsurvey}
	C.~Li and F.~Zhang.
	\newblock A survey on the stability of fractional differential equations.
	\newblock {\em The European Physical Journal Special Topics}, 193(1):27--47,
	2011.
	
	\bibitem{podlubny1999fractional}
	I.~Podlubny.
	\newblock {\em Fractional Differential Equations. An Introduction to Fractional
		Derivatives, Fractional Differential Equations, Some Methods of Their
		Solution and Some of Their Applications}.
	\newblock Academic Press, San Diego - New York - London, 1999.
	
	\bibitem{Stamova2017}
	I.~Stamova, J.~Alzabut, and G.~Stamov.
	\newblock Fractional dynamical systems: Recent trends in theory and
	applications.
	\newblock {\em The European Physical Journal Special Topics},
	226(16):3327--3331, Dec 2017.
	
	\bibitem{newworld}
	H.~Sun, Y.~Zhang, D.~Baleanu, W.~Chen, and Y.~Chen.
	\newblock A new collection of real world applications of fractional calculus in
	science and engineering.
	\newblock {\em Communications in Nonlinear Science and Numerical Simulation},
	64, 04 2018.
	
\end{thebibliography}
\end{document}